\documentclass{amsart}

\usepackage[latin1]{inputenc}
\usepackage[english]{babel}
\usepackage{indentfirst}
\usepackage{amssymb}
\usepackage{amsthm}
\usepackage{xcolor}
\usepackage[all]{xy}

\newcommand{\N}{\mathbb{N}}

\newcommand{\sub}{\subseteq}
\def\epsilon{\varepsilon}
\newtheorem{theo}{Theorem}
\newtheorem{lem}[theo]{Lemma}
\newtheorem{cor}[theo]{Corollary}
\newtheorem{rem}[theo]{Remark}
\newtheorem{exa}[theo]{Example}

\numberwithin{equation}{section}

\title{On control measures of multimeasures}

\author{Jos\'{e} Rodr\'{i}guez}
\address{Dpto. de Ingenier\'{i}a y Tecnolog\'{i}a de Computadores\\Facultad de Inform\'{a}tica\\
Universidad de Murcia\\ 30100 Espinardo (Murcia)\\ Spain} \email{joserr@um.es}


\thanks{Research partially supported by {\em Agencia Estatal de Investigaci\'{o}n} [MTM2017-86182-P, grant cofunded by ERDF, EU] 
and {\em Fundaci\'on S\'eneca} [20797/PI/18]}

\begin{document}

\begin{abstract}
Let $M$ be a multimeasure defined on a $\sigma$-algebra and taking values in the family of bounded non-empty subsets
of a Banach space~$X$. We prove that $M$ admits a control measure whenever $X$ contains no subspace isomorphic to~$c_0(\omega_1)$.
The additional assumption on~$X$ is shown to be essential.
\end{abstract}

\maketitle

\section{Introduction}

Throughout this paper $(\Omega,\Sigma)$ is a measurable space and $X$ is a real Banach space. 
The dual of~$X$ is denoted by~$X^*$ and $B_{X^*}$ is the closed unit ball of~$X^*$.
By a subspace of a Banach space we mean a norm-closed linear subspace.
Given any $A\in \Sigma$, we write $\Sigma_A:=\{B \in \Sigma: B \sub A\}$. 
The set of all countably additive real-valued measures on~$\Sigma$
is denoted by ${\rm ca}(\Sigma)$. The subset of ${\rm ca}(\Sigma)$ consisting of all finite non-negative
measures on~$\Sigma$ is denoted by ${\rm ca}^+(\Sigma)$. Given any $\nu\in {\rm ca}(\Sigma)$, its variation
is denoted by~$|\nu|$. We write $\mathcal{N}(\mu):=\{A\in \Sigma:\mu(A)=0\}$ for every $\mu \in {\rm ca}^+(\Sigma)$. 
The first uncountable ordinal is denoted by~$\omega_1$.

A classical result of Bartle, Dunford and Schwartz~\cite{bar-alt} (cf. \cite[p.~14, Corollary~6]{die-uhl-J}) states that any
countably additive vector measure $m:\Sigma\to X$ admits a control measure, i.e., 
there is $\mu\in {\rm ca}^+(\Sigma)$ such that $m(A)=0$ whenever
$A\in \mathcal{N}(\mu)$. This result is of great importance in the theory of vector measures.
In this paper we discuss the existence of control measures of \emph{set-valued} measures. Among all different concepts of 
set-valued measure (see, e.g., \cite[Chapter~8]{hu-pap}), we work with the weakest one, namely, that of multimeasure.
 
To recall the definition of multimeasure we need to introduce some notation. We denote by $b(X)$ the family of all bounded non-empty subsets of~$X$.
Given $C \in b(X)$ and $x^*\in X^*$, we write 
$$
	s(x^*,C):=\sup\{x^*(x): \, x\in C\}.
$$
A map $M: \Sigma \to b(X)$
is said to be a {\em multimeasure} if, for each $x^*\in X^*$, the map $s(x^*,M): \Sigma\to \mathbb R$
defined by
$$ 
	s(x^*,M)(A):=s(x^*,M(A)) \quad\text{for all $A\in \Sigma$}
$$
belongs to~${\rm ca}(\Sigma)$. Given a map $M:\Sigma \to b(X)$ (non necessarily a multimeasure),
we say that $\mu\in {\rm ca}^+(\Sigma)$ is a {\em control measure} of~$M$ if $M(A)=\{0\}$ whenever $A\in \mathcal{N}(\mu)$ or, equivalently, 
$\mathcal{N}(\mu)$ is contained in
$$
	\mathcal{N}(M):=\{A\in \Sigma: \, M(B)=\{0\} \ \text{for all $B\in \Sigma_A$}\}.
$$

Only a few scattered results on control measures of set-valued measures can be found in the literature, some of them
in the more general setting of locally convex spaces, see \cite[Section~8]{dre} and \cite[Section~3]{dip-alt}.
We highlight that a multimeasure $M:\Sigma \to b(X)$ admits a control measure in each of the following cases:
\begin{enumerate}
\item[(i)] if $M(A)$ is relatively weakly compact for every $A\in \Sigma$;
\item[(ii)] if $X^*$ is weak$^*$-separable (see \cite[Theorem~3.6]{dip-alt}).
\end{enumerate}
On the one hand, case~(i) follows from the result of Bartle, Dunford and Schwartz and the fact that, in this case, the (single-valued) map 
$m: \Sigma \to \ell_\infty(B_{X^*})$ defined by
$$
	m(A)(x^*):=s(x^*,M(A))
	\quad
	\text{for all $A\in \Sigma$ and $x^*\in B_{X^*}$}
$$
is a countably additive vector measure (see, e.g., \cite[p.~852, Theorem~4.10]{hu-pap}; cf. \cite[Theorem~3.4]{cas-kad-rod-2}). On the other hand,
case~(ii) was obtained in~\cite{dip-alt} via a countable chain condition property which characterizes multimeasures admitting control measures 
(see \cite[Theorem~3.4]{dip-alt}). That characterization was inspired by some ideas of~\cite{dre4} and~\cite{mus14}. 
Theorem~\ref{theo:CCC} below is a slight variant of it and we include
a proof in Section~\ref{section:CCC} for completeness. 

A set $\mathcal{A} \sub \Sigma$ is said to have the {\em countable chain condition (CCC)} if 
there is no uncountable subset of~$\mathcal{A}$ consisting of pairwise disjoint sets. Clearly, $\Sigma\setminus \mathcal{N}(\mu)$ has the CCC
for any $\mu\in {\rm ca}^+(\Sigma)$. Therefore, if a map $M:\Sigma \to b(X)$
admits a control measure, then $\Sigma \setminus \mathcal{N}(M)$ has the CCC.

\begin{theo}\label{theo:CCC}
Let $M:\Sigma \to b(X)$ be a map for which there is a symmetric set $\Gamma \sub X^*$ separating the points of~$X$
such that $s(x^*,M) \in {\rm ca}(\Sigma)$ for every $x^*\in \Gamma$. 
Then $M$ admits a control measure if and only if $\Sigma \setminus \mathcal{N}(M)$ has the CCC.
\end{theo}

Theorem~\ref{theo:CCC} is one of the keys to obtain our main result: 

\begin{theo}\label{theo:Main}
Suppose $X$ contains no subspace isomorphic to~$c_0(\omega_1)$. Let $Z \sub X^*$ be a norming subspace.
Let $M:\Sigma \to b(X)$ be a map such that $s(x^*,M) \in {\rm ca}(\Sigma)$ for every $x^*\in Z$. 
Then $M$ admits a control measure.
\end{theo}

Observe that the result considered in case~(ii) above is a consequence of the following corollary (see Remark~\ref{rem:Applications}).

\begin{cor}\label{cor:FullDual}
Suppose $X$ contains no subspace isomorphic to~$c_0(\omega_1)$. Then every multimeasure $M:\Sigma \to b(X)$ 
admits a control measure.
\end{cor}

In \cite[Question~3.4]{dip-alt} it was asked whether, for an arbitrary Banach space~$X$, every multimeasure $M:\Sigma \to b(X)$ admits
a control measure. We provide a negative answer for $X=c_0(\omega_1)$ in Example~\ref{exa:Counterexample}, thus 
showing that the converse of Corollary~\ref{cor:FullDual} holds for some measurable spaces.

The proof of Theorem~\ref{theo:Main}, Example~\ref{exa:Counterexample} and some remarks 
on Banach spaces not containing isomorphic copies of~$c_0(\omega_1)$ are included
in Section~\ref{section:Main}.

\section{Proof of Theorem~\ref{theo:CCC}}\label{section:CCC}

To deal with Theorem~\ref{theo:CCC} we need some lemmas.

\begin{lem}\label{lem:1}
Let $\mathcal{N}_0 \sub \mathcal{N}_1 \sub \Sigma$ be sets such that
$\emptyset \in \mathcal{N}_0$ and $\mathcal{N}_1$ is closed under countable unions.
If $\Sigma \setminus \mathcal{N}_0$ has the CCC, then for every $A\in \Sigma \setminus \mathcal{N}_1$ there is $B\in \Sigma_A \setminus \mathcal{N}_1$
such that $\mathcal{N}_0\cap \Sigma_B=\mathcal{N}_1\cap \Sigma_B$.
\end{lem}
\begin{proof}
By contradiction, suppose that there is $A\in \Sigma \setminus \mathcal{N}_1$ such that 
\begin{equation}\label{eqn:contrad}
	\Sigma_B \cap (\mathcal{N}_1 \setminus \mathcal{N}_0) \neq \emptyset
	\quad
	\text{for every $B\in \Sigma_A \setminus \mathcal{N}_1$.}
\end{equation}
In particular, $\Sigma_A \cap (\mathcal{N}_1 \setminus \mathcal{N}_0) \neq \emptyset$. Zorn's lemma ensures
the existence of a maximal family $\mathcal{C}$ of pairwise disjoint elements of $\Sigma_A\cap (\mathcal{N}_1 \setminus \mathcal{N}_0)$.
Since $\Sigma \setminus \mathcal{N}_0$ has the CCC, the family $\mathcal{C}$ is countable, say $\mathcal{C}=\{C_1,C_2,\dots\}$ (maybe finite), 
and therefore $\bigcup_n C_n\in \Sigma_A \cap \mathcal{N}_1$.
Since $A\not\in \mathcal{N}_1$, we have $A\setminus \bigcup_n C_n \in \Sigma_A \setminus \mathcal{N}_1$.
By~\eqref{eqn:contrad} applied to $A\setminus \bigcup_n C_n$, there is $B' \in \Sigma_A \cap (\mathcal{N}_1\setminus \mathcal{N}_0)$ with
$B' \sub A\setminus \bigcup_n C_n$. This contradicts the maximality of~$\mathcal{C}$
(note that $\emptyset\not\in \Sigma_A \cap (\mathcal{N}_1 \setminus \mathcal{N}_0)$).
\end{proof}

\begin{lem}\label{lem:2}
Let $\{\mathcal{N}_i\}_{i\in I}$ be a family of non-empty subsets of~$\Sigma$ such that each $\mathcal{N}_i$ 
is closed under countable unions and satisfies $\Sigma_A \sub \mathcal{N}_i$ for every $A\in \mathcal{N}_i$.
If $\Sigma \setminus \bigcap_{i\in I}\mathcal{N}_i$ has the CCC, then there is a countable set
$I_0 \sub I$ such that 
$$
	\bigcap_{i\in I}\mathcal{N}_i=\bigcap_{i\in I_0}\mathcal{N}_i.
$$
\end{lem}
\begin{proof}
Write $\mathcal{N}:=\bigcap_{i\in I}\mathcal{N}_i$. Note that $\emptyset \in \mathcal{N}$.
The statement is obvious if $\mathcal{N}=\Sigma$, so we assume that $\mathcal{N}\neq\Sigma$. 
Define 
$$
	\mathcal{B}:=\{B\in \Sigma \setminus \mathcal{N}: \ \mathcal{N}\cap \Sigma_B=\mathcal{N}_i\cap \Sigma_B
	\text{ for some $i\in I$}\}. 
$$
Observe that $\mathcal{B}\neq \emptyset$. Indeed, take $i\in I$ for which $\mathcal{N}_i\neq \Sigma$ 
and pick any $A\in \Sigma \setminus \mathcal{N}_i$. By Lemma~\ref{lem:1}, there is 
$B\in \Sigma_A \setminus \mathcal{N}_i$
such that $\mathcal{N}\cap \Sigma_B=\mathcal{N}_i\cap \Sigma_B$, so $B\in \mathcal{B}$.

By Zorn's lemma, there is a maximal family $\mathcal{B}_0$ of pairwise disjoint elements of~$\mathcal{B}$. From the fact that 
$\Sigma\setminus \mathcal N$ has the CCC it follows that $\mathcal{B}_0$ is countable, say $\mathcal{B}_0=\{B_1,B_2,\dots\}$. 
For each~$n$ we choose $i(n)\in I$ such that 
\begin{equation}\label{eqn:piece}
	\mathcal{N}\cap \Sigma_{B_n}=\mathcal{N}_{i(n)}\cap \Sigma_{B_n}.
\end{equation}
To finish the proof, we will show that
$$
	\mathcal{N}=\bigcap_{n}\mathcal{N}_{i(n)}.
$$
Fix $A\in \bigcap_{n}\mathcal{N}_{i(n)}$. Observe that
$A \cap \bigcup_n B_n \in \mathcal{N}$. Indeed, 
for each~$n$, we have $A\in \mathcal{N}_{i(n)}$ and so $A \cap B_n\in \mathcal{N}_{i(n)}$,
hence~\eqref{eqn:piece} guarantees that $A\cap B_n \in \mathcal{N}$. Since $\mathcal{N}$ is closed under countable unions,
we conclude that $A \cap \bigcup_n B_n \in \mathcal{N}$, as claimed. So, in order to show 
that $A \in \mathcal{N}$ it only remains to check that $A \setminus \bigcup_n B_n \in \mathcal{N}$. By contradiction, suppose
that $A \setminus \bigcup_n B_n \not \in \mathcal{N}_i$ for some $i\in I$. Now, we can apply Lemma~\ref{lem:1}
to get $B\in \Sigma \setminus \mathcal{N}_i$ such that $B \sub A \setminus \bigcup_n B_n$ 
and $\mathcal{N}\cap \Sigma_B=\mathcal{N}_i\cap \Sigma_B$. Therefore, $B$ is disjoint from 
each element of~$\mathcal{B}_0$ and belongs to~$\mathcal{B}$. This contradicts
the maximality of~$\mathcal{B}_0$ (note that $\emptyset \not\in \mathcal{B}$). 
\end{proof}

\begin{lem}\label{lem:3}
Let $S \sub {\rm ca}^+(\Sigma)$ be a non-empty set. If $\Sigma \setminus \bigcap_{\mu \in S}\mathcal{N}(\mu)$ has the CCC, then
there is $\mu_0 \in {\rm ca}^+(\Sigma)$ such that 
$$
	\bigcap_{\mu \in S}\mathcal{N}(\mu)=\mathcal{N}(\mu_0).
$$
\end{lem}
\begin{proof} We can suppose without loss of generality
that $\mu(\Omega)=1$ for every $\mu\in S$. By Lemma~\ref{lem:2}, there is a countable set $S_0\sub S$
such that $\bigcap_{\mu \in S}\mathcal{N}(\mu)=\bigcap_{\mu \in S_0}\mathcal{N}(\mu)$.
Enumerate $S_0=\{\mu_1,\mu_2,\dots\}$ and define $\mu_0\in {\rm ca}^+(\Sigma)$
by $\mu_0(A):=\sum_n 2^{-n}\mu_n(A)$ for all $A\in \Sigma$. Then $\bigcap_{\mu \in S}\mathcal{N}(\mu)=\mathcal{N}(\mu_0)$.
\end{proof}

Theorem~\ref{theo:CCC} can now be obtained from Lemma~\ref{lem:3}, as follows:

\begin{proof}[Proof of Theorem~\ref{theo:CCC}]
Write $\mu_{x^*}:=|s(x^*,M)|\in {\rm ca}^+(\Sigma)$ for every $x^*\in \Gamma$.

We claim that 
\begin{equation}\label{eqn:NM}
	\mathcal{N}(M)=\bigcap_{x^*\in \Gamma}\mathcal{N}(\mu_{x^*}). 
\end{equation}
Indeed, if $A \in \mathcal{N}(M)$, then
for each $x^*\in \Gamma$ we have $s(x^*,M(B))=0$ for every $B \in \Sigma_A$, hence $A\in \mathcal{N}(\mu_{x^*})$.  
Conversely, take $A\in \Sigma \setminus \mathcal{N}(M)$. Then there is $B \in \Sigma_A$ such that $M(B)\neq \{0\}$.
Fix $x\in M(B) \setminus \{0\}$ and choose $x^*\in \Gamma$ such that $x^*(x)>0$. Then $s(x^*,M(B))\geq x^*(x)>0$
and therefore $\mu_{x^*}(B)>0$. Hence $A\not\in \mathcal{N}(\mu_{x^*})$. This proves~\eqref{eqn:NM}.

Now, if $\Sigma\setminus \mathcal{N}(M)$ has the CCC, then equality~\eqref{eqn:NM} and Lemma~\ref{lem:3} ensure the existence of
$\mu \in {\rm ca}^+(\Sigma)$ such that $\mathcal{N}(M)=\mathcal{N}(\mu)$, so $\mu$ is a control measure of~$M$. 
\end{proof}

\section{Main results}\label{section:Main}

The proof of Theorem~\ref{theo:Main} uses Lemma~\ref{lem:Disjoint} below. 
Given a subspace $Z \sub X^*$, we denote by $j_Z: X \to Z^*$ the bounded linear operator defined by 
$$
	j_Z(x)(x^*):=x^*(x) \quad\text{for all $x\in X$ and $x^*\in Z$.}
$$
As usual, given a non-empty set $I$, for each $i\in I$ we denote by $e_i$ the element of~$c_0(I)$ defined by $e_i(j)=0$ for all $j\in I\setminus \{i\}$ and $e_i(i)=1$.

\begin{lem}\label{lem:Disjoint}
Let $Z \sub X^*$ be a subspace and $M:\Sigma \to b(X)$ be a map such that $s(x^*,M)\in {\rm ca}(\Sigma)$ for every $x^*\in Z$. 
Let $\Delta$ be a non-empty set of pairwise disjoint elements of~$\Sigma$ and let
$x_A\in M(A)$ for every $A\in \Delta$. Then there is a bounded linear operator $S: \ell_\infty(\Delta) \to Z^*$
such that $S(e_A)=j_Z(x_A)$ for every $A\in \Delta$.
\end{lem}
\begin{proof} We claim that for every $x^*\in Z$ we have $(x^*(x_A))_{A\in \Delta}\in \ell_1(\Delta)$, with 
\begin{equation}\label{eqn:norm-l1}
	\sum_{A \in \Delta} |x^*(x_A)| \leq |s(x^*,M)|(\Omega)+|s(-x^*,M)|(\Omega).
\end{equation}
Indeed, take any finite set $\Delta' \sub \Delta$ and
write $\Delta'_{+}:=\{A\in \Delta': x^*(x_A)\geq 0\}$ and $\Delta'_{-}:=\Delta' \setminus \Delta'_{+}$. Then
\begin{eqnarray*}
	\sum_{A \in \Delta'} |x^*(x_A)| 
	 & = & \sum_{A \in \Delta'_{+}} x^*(x_A) + \sum_{A \in \Delta'_{-}} (-x^*)(x_A) 
	\\ &\leq& \sum_{A\in \Delta'_{+}} s(x^*,M(A)) + \sum_{A\in \Delta'_{-}} s(-x^*,M(A))
	\\ &\leq& |s(x^*,M)|(\Omega)+|s(-x^*,M)|(\Omega).
\end{eqnarray*}
As $\Delta'$ is an arbitrary finite subset of~$\Delta$, we conclude that $(x^*(x_A))_{A\in \Delta}\in \ell_1(\Delta)$
and that inequality~\eqref{eqn:norm-l1} holds.

Now, we define a linear map $T:Z \to \ell_1(\Delta)$ by 
$$
	T(x^*):=(x^*(x_A))_{A\in \Delta} \quad \text{for all $x^*\in Z$.} 
$$
An appeal to the Closed Graph Theorem ensures that $T$ is continuous. Clearly,
its adjoint $S:=T^*: \ell_\infty(\Delta) \to Z^*$ satisfies the required property.
\end{proof}

Recall that a subspace $Z \sub X^*$ is called {\em norming} if the formula 
$$
	|||x|||=\sup\{|x^*(x)|:\, x^*\in Z \cap B_{X^*}\} \quad \text{for all $x\in X$}
$$
defines an equivalent norm on~$X$, that is, $j_Z$ is an isomorphic embedding. 

\begin{proof}[Proof of Theorem~\ref{theo:Main}]
Since $Z$ is a norming subspace of~$X^*$, it is symmetric and separates the points of~$X$. So, by Theorem~\ref{theo:CCC}, $M$ admits a control measure if (and only if) 
$\Sigma \setminus \mathcal{N}(M)$ has the CCC. We will prove that if $\Sigma \setminus \mathcal{N}(M)$ fails the CCC, 
then $X$ contains a subspace isomorphic to $c_0(\omega_1)$, thus contradicting the assumption.
 
Let $\Delta_0$ be an uncountable set of pairwise disjoint elements of~$\Sigma \setminus \mathcal{N}(M)$. 
We can (and do) assume without loss of generality that $M(A) \neq \{0\}$ for every $A\in \Delta_0$.  Indeed,
for each $A\in \Delta_0$ there is $B_A\in \Sigma_A$ such that $M(B_A)\neq \{0\}$. Each $B_A$ is non-empty 
(by equality \eqref{eqn:NM} in the proof of Theorem~\ref{theo:CCC})
and $B_A \cap B_{A'} =\emptyset$ whenever $A\neq A'$. Thus, we might replace $\Delta_0$ with $\{B_A:A\in \Delta_0\}$ if necessary.

Fix $x_A\in M(A)\setminus \{0\}$ for every $A\in \Delta_0$.
Since $j_Z$ is an isomorphic embedding, it is injective and so $j_Z(x_A)\neq 0$ for every $A\in \Delta_0$.
Since $\Delta_0$ is uncountable, there exist an uncountable set $\Delta \sub \Delta_0$
and $\epsilon>0$ such that $\|j_Z(x_A)\|\geq \epsilon$ for all $A\in \Delta$. 

By Lemma~\ref{lem:Disjoint}, there is a bounded linear operator $S:\ell_\infty(\Delta)\to Z^*$ such that
$S(e_A)=j_Z(x_A)$ for all $A\in \Delta$. In particular, we get $S(c_0(\Delta)) \sub j_Z(X)$
(bear in mind that $j_Z$ has closed range). 
Since $\|S(e_A)\|\geq \epsilon$ for every $A\in \Delta$, a result of Rosenthal \cite{ros-70} (see, e.g., \cite[Theorem~7.10]{fab-alt-JJ}) 
ensures that there is a set $\Delta' \sub \Delta$ with ${\rm card}(\Delta')={\rm card}(\Delta)$ such that $S$ is an 
isomorphic embedding when restricted to $\ell_\infty(\Delta')$ (as a subspace of~$\ell_\infty(\Delta)$).
Therefore, $j_Z(X)$ contains a subspace isomorphic to~$c_0(\Delta')$, and the same holds for~$X$.
Since $\Delta'$ is uncountable, we conclude that $X$ contains a subspace isomorphic to~$c_0(\omega_1)$.
\end{proof}

Theorem~\ref{theo:Main} applies in a natural way to set-valued maps in dual Banach spaces to get the following corollary.  
A map $M:\Sigma \to b(X^*)$ is called a {\em weak$^*$-multimeasure} if $s(x,M) \in {\rm ca}(\Sigma)$ for every $x\in X$. 

\begin{cor}\label{cor:Dual}
Suppose $X^*$ contains no subspace isomorphic to~$c_0(\omega_1)$. 
Then every weak$^*$-multimeasure $M:\Sigma \to b(X^*)$ admits a control measure.
\end{cor}

We stress that if $X^*$ contains no subspace isomorphic to $c_0$, then every weak$^*$-multimeasure whose 
values are convex $w^*$-compact subsets of~$X^*$ is a multimeasure, see \cite[Theorem~3.4]{mus13}.

Let us present an example of a multimeasure not admitting a control measure. Note that, for instance,
the construction can be carried out for the Borel $\sigma$-algebra of any uncountable Hausdorff topological space.

\begin{exa}\label{exa:Counterexample}
Suppose there is an uncountable set $\Omega_0 \sub \Omega$ such that $\{t\}\in \Sigma$ for every~$t\in \Omega_0$. 
Fix an injective map $i:\omega_1\to \Omega_0$ and define $M:\Sigma \to b(c_0(\omega_1))$ by 
$$
	M(A):=\left\{\sum_{\alpha \in F} e_\alpha-\sum_{\alpha \in G}e_\alpha: \ \text{$F,G \sub i^{-1}(A)$ are finite and disjoint}\right\}
$$
for every $A\in \Sigma$. Then $M$ is a multimeasure not admitting a control measure. 
\end{exa}
\begin{proof} For each $\alpha <\omega_1$ we have 
$e_\alpha\in M(\{i(\alpha)\})$, hence $\{i(\alpha)\}\in \Sigma \setminus \mathcal{N}(M)$. 
Since $i$ is injective, we have $i(\alpha) \neq i(\alpha')$ whenever $\alpha\neq \alpha'$.
Therefore, $\Sigma\setminus \mathcal{N}(M)$ fails the CCC and so $M$ does not admit a control measure. 

To check that $M$ is a multimeasure, take 
$\varphi=(\varphi(\alpha))_{\alpha<\omega_1} \in \ell_1(\omega_1) = c_0(\omega_1)^*$. For each $A\in \Sigma$ we have
\begin{eqnarray*}
	s(\varphi,M(A)) &=&\sup\left\{ \sum_{\alpha\in F}\varphi(\alpha)-\sum_{\alpha \in G}\varphi(\alpha): \ 
	\text{$F,G \sub i^{-1}(A)$ finite, $F\cap G=\emptyset$}\right\}\\
	 &=& \sum_{\alpha\in i^{-1}(A)} |\varphi(\alpha)|.
\end{eqnarray*}
Therefore, if $(A_n)_{n\in \N}$ is a sequence of pairwise disjoint elements of~$\Sigma$, then
\begin{eqnarray*}
	s\left(\varphi,M\left(\bigcup_{n\in \N}A_n\right)\right)&=&
	\sum_{\alpha\in i^{-1}(\bigcup_{n\in \N}A_n)}|\varphi(\alpha)| \\ &=&
	\sum_{n\in \N}\sum_{\alpha\in i^{-1}(A_n)}|\varphi(\alpha)| = \sum_{n\in \N}s(\varphi,M(A_n)).
\end{eqnarray*}
It follows that $s(\varphi,M)\in {\rm ca}(\Sigma)$. 
\end{proof}

\begin{cor}\label{cor:equivalence}
Suppose $\{t\}\in \Sigma$ for uncountably many~$t\in \Omega$. The following statements are equivalent:
\begin{enumerate} 
\item[(i)] Every multimeasure $M: \Sigma \to b(X)$ admits a control measure.
\item[(ii)] $X$ contains no subspace isomorphic to~$c_0(\omega_1)$.
\end{enumerate}
\end{cor}

We finish the paper with some remarks on Banach spaces not containing subspaces isomorphic to~$c_0(\omega_1)$.

\begin{rem}\label{rem:CK}
Let $K$ be a compact Hausdorff topological space. Then $C(K)$ contains no subspace isomorphic to~$c_0(\omega_1)$ if and only if 
$K$ has the CCC (i.e., every
family of pairwise disjoint open subsets of~$K$ is countable), according
to a result of Rosenthal~\cite{ros-J-4} (see, e.g., \cite[Theorem~7.22]{fab-alt-JJ}).
\end{rem}

In what follows, $B_{X^*}$ is equipped with the weak$^*$-topology.

\begin{rem}\label{rem:Applications}
In general, a Banach space~$X$ does not contain subspaces isomorphic to~$c_0(\omega_1)$ in each of the following cases: 
\begin{enumerate}
\item[(a)] if $B_{X^*}$ has the CCC;
\item[(b)] if $X^*$ is weak$^*$-separable.
\end{enumerate}
Case~(a) is a consequence of the result mentioned in Remark~\ref{rem:CK} and the canonical embedding of~$X$ 
into~$C(B_{X^*})$. For case~(b), just observe that the property of having weak$^*$-separable dual is inherited by subspaces
and that $c_0(\omega_1)^*=\ell_1(\omega_1)$ is not weak$^*$-separable.
\end{rem}

There are Banach spaces~$X$ such that $B_{X^*}$ has the CCC and $X^*$ is not weak$^*$-separable:

\begin{exa}\label{exa:l1Gamma}
Let $X=\ell_1(\Gamma)$ for a non-empty set~$\Gamma$. Then 
$B_{X^*}$ is homeomorphic to the Tychonoff cube $[0,1]^\Gamma$, which has the CCC (see, e.g., \cite[2.3.18]{eng-J}). 
However, $X^*=\ell_\infty(\Gamma)$ is not weak$^*$-separable whenever ${\rm card}(\Gamma)>\mathfrak{c}$ (the continuum). Indeed,
if a Banach space $Y$ has weak$^*$-separable dual, then there is an injective bounded linear operator from~$Y$ to~$\ell_\infty$, 
hence ${\rm card}(Y) \leq {\rm card}(\ell_\infty)=\mathfrak{c}$. 
\end{exa}

We do not know whether the weak$^*$-separability of~$X^*$ implies that $B_{X^*}$ has the CCC. The answer is affirmative for $C(K)$ spaces:

\begin{rem}\label{rem:CK2}
Let $K$ be a compact Hausdorff topological space. If $C(K)^*$ is weak$^*$-separable, then $B_{C(K)^*}$ has the CCC.
\end{rem}
\begin{proof}
Since $C(K)^*$ is weak$^*$-separable, there is an injective bounded linear operator $T: C(K) \to \ell_2$. Then
$T^*(B_{\ell_2})$ is a weakly compact subset of~$C(K)^*$ which separates the points of~$C(K)$. Another result of Rosenthal (see \cite[Theorem~4.5(b)]{ros-J-4}) ensures that $K$ carries a strictly positive measure (i.e., there is a regular Borel probability $\mu$ on~$K$
such that $\mu(U)>0$ whenever $U \sub K$ is open and non-empty). A folklore argument (cf. \cite{ple-thesis}) applies
to conclude that $B_{C(K)^*}$ carries a strictly positive measure as well, which clearly implies that
$B_{C(K)^*}$ has the CCC. 
\end{proof}

For a compact Hausdorff topological space~$K$, it is known that 
$B_{C(K)^*}$ has the CCC if and only if $K^n$ has the CCC for every $n\in \N$  (see \cite[Theorem~12]{whe3}).
Under MA$+\neg$CH, these conditions are equivalent to the fact that $K$ has the CCC. 
However, under~CH, there exist
compact spaces having the CCC whose square fails it. For more information on this see, e.g.,~\cite{tod-J-4}.

\subsection*{Acknowledgements}
The author wishes to thank A.~Avil\'{e}s and G.~Plebanek for helpful comments.


\bibliographystyle{amsplain}

\end{document}